\documentclass[10pt]{article}
\usepackage[all]{xy}
\usepackage{amsfonts,amsmath,oldgerm,amssymb,amscd}
\newcommand{\ra}{\rightarrow}		\newcommand{\lra}{\longrightarrow}
\newcommand{\by}[1]{\stackrel{#1}{\ra}}
\newcommand{\lby}[1]{\stackrel{#1}{\lra}}

\newcommand{\surj}{\ra\!\!\!\ra}	\newcommand{\inj}{\hookrightarrow}
\newcommand{\ol}{\overline}		\newcommand{\wt}{\widetilde}
\newcommand{\iso}{\by \sim}

\newtheorem{theorem}{Theorem}[section]
\newtheorem{proposition}[theorem]{Proposition}
\newtheorem{lemma}[theorem]{Lemma}
\newtheorem{corollary}[theorem]{Corollary}

\newcommand{\ga}{\alpha}	\newcommand{\gb}{\beta}
		\newcommand{\gd}{\delta}
	
\newcommand{\gj}{\blacksquare}	\newcommand{\gl}{\lambda}

\newcommand{\gt}{\theta}

	\newcommand{\gD}{\Delta}

\newcommand{\ot}{\mbox{\,$\otimes$\,}}	\newcommand{\op}{\mbox{$\oplus$}}
\newcommand{\Spec}{\mbox{\rm Spec\,}}

\newcommand{\Um}{\mbox{\rm Um}}		\newcommand{\SL}{\mbox{\rm SL}}

\begin{document}

\begin{center}
{\bf \Large Some results on Euler class groups }
\vspace{.2in} \\
	{\large Manoj  K. Keshari} 
\vspace{.1in}\\
{\small 
Department of Mathematics, IIT Mumbai, Mumbai - 400076, India;\;
	keshari@math.iitb.ac.in} 
\

\end{center}

{\small
\noindent {\bf Abstract:} Let $A$ be a regular domain of dimension $d$
containing an infinite field and let $n$ be an integer with $2n\geq
d+3$. For a stably free $A$-module $P$ of rank $n$, we prove that
$(i)$ $P$ has a unimodular element if and only if the euler class of
$P$ is zero in $E^n(A)$ and $(ii)$ we define
Whitney class homomorphism $w(P):E^s(A)\ra E^{n+s}(A)$, where $E^s(A)$
denotes the $s$th Euler class group of $A$ for $s\geq 1$.
\vspace{.1in}

\vspace{.1in}

}

\section{Introduction}

Let $A$ be a commutative Noetherian ring of dimension $d$. For any
$1\leq s \leq d$, abelian group called the Euler class group
$E^s(A)$ of $A$ is defined in \cite{SY} and given any projective $A$-module
$P$ of rank $n<d$, a Whitney class homomorphism $w(P):E^{d-n} \ra
E^d(A)$ is defined. Further it is proved that if $P$ has a unimodular
element, then $w(P)$ is the zero map.
Assume that $A$ is a regular domain of dimension $d$ containing an infinite
field $k$. For a positive integer $n$ with $2n \geq d+3$, 
we prove the following results:

$(i)$ For a stably free $A$-module $P$ of rank $n$, we will associate an
element $e(P)$ of $E^n(A)$ and prove that $e(P)=0$ in $E^n(A)$ if and
only if $P$ splits off a free summand of rank one (i.e. $P=Q\op A$ for
some projective $A$-module $Q$ of rank $n-1$). When $P\op
A\iso A^{n+1}$, this result is due to Bhatwadekar and Raja Sridharan
\cite{Bh-Raja-5}.

$(ii)$ An element $(J,w_J)$ is zero in $E^n(A[T])$ if and only if $J$
is generated by $n$ elements and $w_J$ is a global orientation of
$J$. This result is also proved by Das and Raja (\cite{DR}, Theorem
3.1), but their proof is different from ours.  When $2n>d+3$, this
result follows from \cite{Bh-Raja-5} for any Noetherian ring
$A$. Hence the regularity of the ring is used only in the case
$2n=d+3$.

$(iii)$ Given a stably free $A$-module $Q$ of rank $n$, we define a
Whitney class homomorphism $w(Q):E^s(A) \ra E^{n+s}(A)$. Further, we
prove that if $Q$ has a unimodular element, then $w(Q)$ is the zero
map. When $n+s=d$, these results are proved in \cite{SY} for arbitrary
projective module $Q$ over any Noetherian ring $A$.  

It will be ideal to define the Whitney class homomorphism for all
projective $A$-module $Q$ of rank $n$. For this first we need to define the
euler class of $Q$ in $E^n(A)$ which is not known.


\section{Euler class groups}
All the rings considered are commutative Noetherian and all the
modules are finitely generated.  For a ring $A$ of dimension $d\geq 2$
and $1\leq n\leq d$, the $n$th Euler class group of $A$, denoted by
$E^n(A)$ is defined in \cite{SY} as follows:

Let $E_n(A)$ denote the group generated by $n\times n$ elementary
matrices over $A$ and let $F=A^n$. 
A local orientation is a pair $(I,w)$, where $I$ is an ideal of $A$ of
height $n$ and $w$ is an equivalence class of surjective homomorphisms
from $F/IF$ to $I/I^2$. The equivalence is defined by $E_n(A/I)$-maps.

Let $L^n(A)$ denote the set of all pairs $(I,w)$, where $I$ is an
ideal of height $n$ such that $\Spec (A/I)$ is connected and $w:F/IF
\surj I/I^2$ is a local orientation.  Similarly, let $L_0^n(A)$
denote the set of all ideals $I$ of height $n$ such that $\Spec
(A/I)$ is connected and there is a surjective homomorphism from $F/IF$
to $I/I^2$.

Let $G^n(A)$ denote the free abelian group generated by $L^n(A)$ and
let $G_0^n(A)$ denote the free abelian group generated by $L_0^n(A)$.

Suppose $I$ is an ideal of height $n$ and $w:F/IF \surj I/I^2$ is a
local orientation. By (\cite{Bh-Raja-5}, Lemma 4.1), there is a unique
decomposition $I=\cap_1^r I_i$, such that $I_i$'s are pairwise
comaximal ideals of height $n$ and $\Spec (A/I_i)$ is connected.  Then
$w$ naturally induces local orientations $w_i:F/I_iF \surj I_i/I_i^2$.
Denote $(I,w)=\sum (I_i,w_i) \in G^n(A)$. Similarly we denote
$(I)=\sum (I_i) \in G_0^n(A)$.

We say a local orientation $w:F/IF\surj I/I^2$ is global if $w$ can be
lifted to a surjection $\Omega : F\surj I$.
Let $H^n(A)$ be the subgroup of $G^n(A)$ generated by global
orientations. Also let $H_0^n(A)$ be the subgroup of $G_0^n(A)$
generated by $(I)$ such that $I$ is a surjective image of $F.$

The Euler class group of codimension $n$ cycles is defined as
$E^n(A)=G^n(A)/H^n(A)$ and the weak Euler class group of codimension
$n$ cycles is defined as $E_0^n(A)=G_0^n(A)/H_0^n(A)$.

The following result is proved in (\cite{Bh-Raja-5}, Corollary 2.4) in
the case $P$ is free. Same proof works in this case, hence we
omit the proof.

\begin{lemma}\label{moving}
Let $A$ be a  ring of dimension $d$ and let $n$ be an integer such that
$2n \geq d+1$. Let $I$ be an ideal of $A$ of 
height $n$. Let $P$ be  a projective $A$-module of rank $n$.
Suppose 
$\phi : P \surj I/I^2$ be a surjection. Then, we can find a lift
$\Phi' :P \ra I$ of $\phi$ such that
$\Phi'(P)= I\cap I'$, where $I'$ is an ideal of height $\geq n$ and
comaximal with $I$.

Further, given any ideal $K$ of $A$ of height $\geq d-n+1$, we can
choose $I'$ to be comaximal with $K$.
\end{lemma}

Using (\ref{moving}), (\cite{Bh-Manoj}, 4.11, 5.7) and following the proof of
(\cite{Bh-Raja-5}, Proposition 3.3), we can prove the following
result. Hence we omit the proof.

\begin{proposition}\label{ll}
Let $A$ be a regular domain of dimension $d$ containing an infinite
field and let $n$ be an integer such that $2n\geq d+3$. Let $P=Q\op A$
be a projective $A$-module of rank $n$. Let $J,J_1,J_2$ be ideals of
$A[T]$ of height $n$ such that $J$ is comaximal with $J_1$ and
$J_2$. Assume that there exist surjections $$\ga : P[T] \surj J\cap
J_1,\; \gb: P[T] \surj J\cap J_2$$ with $\ga \ot A[T]/J=\gb\ot
A[T]/J$. Suppose that there exists an ideal $J_3\subset A[T]$ of
height $n$ such that $J_3$ is comaximal with $J,J_1,J_2$ and there
exists a surjection $\gamma:P[T] \surj J_3\cap J_1$ with $\ga \ot
A[T]/J_1=\gamma\ot A[T]/J_1$.

Then there exists a surjection $\delta : P[T] \surj J_3\cap J_2$
with $\delta \ot A[T]/J_3=\gamma\ot A[T]/J_3$ and $\delta \ot
A[T]/J_2=\gb \ot A[T]/J_2$.
\end{proposition}

If we replace $A[T]$ be any
Noetherian ring $B$ of dimension $d$ and $P[T]$ by any projective
$B$-module $\wt P=Q\op B$ of rank $n$, then using (\cite{Bh-Manoj},
Theorems 3.7 and 5.6) and following the proof of (\cite{Bh-Raja-5},
Proposition 3.3), we can prove (\ref{ll}) in this case also.

Using (\cite{Bh-Manoj}, 4.11, 5.7), (\ref{ll}) and following the proof
of (\cite{Bh-Raja-5}, Theorem 4.2), we can prove the following
result. This result is also proved in (\cite{DR}, Theorem 3.1).  Note
that regularity of the ring is used only when $2n=d+3$. When $2n >
d+3$, (\ref{zero}) holds for any ring $A$ by (\cite{Bh-Raja-5},
Theorem 4.2).

\begin{theorem}\label{zero}
Let $A$ be a regular ring of dimension $d\geq 3$ containing an
infinite field and let $n$ be an integer such that $2n\geq
d+3$. Assume that the image of $(J,w_J)$ is zero in $E^n(A[T])$, where
$J\subset A[T]$ is an ideal of height $n$ and $w_J:(A[T]/J)^n \surj
J/J^2$ is an equivalence class of surjections. Then $J$ is generated
by $n$ elements and $w_J$ can be lifted to a surjection $\gt: A[T]^n
\surj J$.
\end{theorem}

\subsection{Euler class of Stably free modules}
 Let $A$ be a regular ring
of dimension $d\geq 3$ containing an infinite field and let $n$ be an integer
such that $2n\geq d+3$. In \cite{Bh-Raja-5}, a map from $\Um_{n+1}(A)$
to $E^n(A)$ is defined and it is proved that, if $P$ is a projective
$A$-module of rank $n$ defined by the unimodular element
$[a_0,\ldots,a_n]$, then $P$ has a unimodular element if and only if
the image of $[a_0,\ldots,a_n]$ in $E^n(A)$ is zero (\cite{Bh-Raja-5},
Theorem 5.4). Note that $P\op A\iso A^{n+1}$.

For $r\geq 1$, let $\Um_{r,n+r}(A)$ be the set of all $r\times (n+r)$
matrices $\sigma$ in $M_{r,n+r}(A)$ which has a right inverse, i.e
there exists $\tau \in M_{n+r,r}$ such that $\sigma \tau$ is the
$r\times r$ identity matrix. For any element $\sigma \in
\Um_{r,n+r}(A)$, we have an exact sequence
$$0\ra A^r \by \sigma A^{n+r} \ra P\ra 0,$$ where $\sigma(v)=v\sigma$
for $v\in A^r$ and $P$ is a stably free projective $A$-module of rank
$n$. Hence, every element of $\Um_{r,n+r}(A)$ corresponds to a stably
free projective $A$-module of rank $n$ and conversely, any stably free
projective $A$-module $P$ of rank $n$ will give rise to an element of
$\Um_{r,n+r}(A)$ for some $r$.  We will define a map from
$\Um_{r,n+r}(A)$ to $E^n(A)$ which is a natural generalization of the
map $\Um_{n+1}(A)\ra E^n(A)$ defined in \cite{Bh-Raja-5}.

Let $\sigma$ be an element of $\Um_{r,n+r}(A)$. 
$$\sigma = \left [
	\begin{array}{ccc}
	{a_{1,1}} & \dots &{a_{1,n+r}} \\
	\vdots     &       & \vdots   \\ 
	a_{r,1} & \ldots &a_{r,n+r}  
	
	\end{array}
\right]$$
Let $e_1,\ldots,e_{n+r}$ be the standard basis of $A^{n+r}$ and let
$$P=A^{n+r}/(\sum_{i=1}^{n+r} a_{1,i}\,e_i,\ldots, \sum_{i=1}^{n+r}
a_{r,i}\,e_i)A.$$ Let $p_1,\ldots,p_{n+r}$ be the images of
$e_1,\ldots,e_{n+r}$ respectively in $P$. Then
$$P=\sum_{i=1}^{n+r} A\,{p_i} \;{\rm with \;relations}\; \sum_{i=1}^{n+r}
a_{1,i}\,p_i=0,\ldots, \sum_{i=1}^{n+r} a_{r,i}\,p_i=0.$$ To the triple
$(P,(p_1,\ldots,p_{n+r}),\sigma)$, we associate an element
$e(P,(p_1,\ldots,p_{n+r}),\sigma)$ of $E^n(A)$ as follows:

Let $\gl:P\surj J$ be a surjection, where $J\subset A$ is an ideal of
height $n$. Since $P\op A^r=A^{n+r}$ and $\dim A/J \leq d-n\leq n-3$,
by \cite{B}, $P/JP$ is a free $A/J$-module of rank
$n$. Since $J/J^2$ is a surjective image of $P/JP$, $J/J^2$ is
generated by $n$ elements.

Let ``bar'' denote reduction modulo $J$. By Bass result
(\cite{B}), there exists $\Theta \in E_{n+r}(A/J)$ such that
$[\ol {a_{1,1}},\ldots, \ol {a_{1,n+r}}]\,\Theta =[1,0,\ldots,0]$. Let
${\ol \sigma}\,\Theta  = 
\left [
	\begin{array}{cccc}
	  1  & 0 &0 & 0 \\
	  \ol {b_{2,1}} &\ol {b_{2,2}}&  \ldots &\ol {b_{2,n+r}}\\
	  \vdots     & \vdots      & \vdots & \vdots  \\ 
	  \ol {b_{r,1}}& \ol {b_{r,2}} & \ldots &b_{r,n+r}  
	\end{array} \right]$.

Further, there exists $\Theta_1\in E_{n+r}(A/J)$ such that ${\ol
\sigma}\,\Theta \,\Theta_1=\left [
	\begin{array}{cccc} 1&0&0&0 \\
	 0&\ol {b_{2,2}}&  \ldots &\ol {b_{2,n+r}}\\
	  \vdots &\vdots      & \vdots & \vdots  \\ 
	  0& \ol {b_{r,2}} & \ldots &b_{r,n+r}  
	\end{array} \right]$.

It is clear that the first row of the elementary matrix
$(\Theta\Theta_1)^{-1}$ is $[\ol {a_{1,1}}, \ldots, \ol {a_{1,n+r}}]$
and the matrix $\ol \sigma_1= \left [
	\begin{array}{ccc}
	 \ol {b_{2,2}}&  \ldots &\ol {b_{2,n+r}}\\
	  \vdots      & \vdots & \vdots  \\ 
	  \ol {b_{r,2}} & \ldots &b_{r,n+r}  
	\end{array} \right]$ belongs to $\Um_{(r-1),(n+r-1)}(A/J)$.
Hence, by induction on $r$, there exists $\Theta_2\in E_{n+r-1}(A/J)$
such that the first $r-1$ rows of $\Theta_2$ are $\ol\sigma_1$. Hence
$\ol \sigma$ can be completed to an elementary matrix $\Delta \in
E_{n+r}(A/J)$ (i.e. $\ol \sigma$ is the first $r$ rows of an
elementary matrix $\Delta \in E_{n+r}(A/J) $).

Since $\sum_{i=1}^{n+r} a_{1,i}\,p_i =0,\ldots ,\sum_{i=1}^{n+r}
a_{r,i}\,p_i=0$, we get
$$\Delta [\ol {p_1},\ldots ,\ol {p_{n+r}}]^t = [0,\ldots,0,\ol
{q_1},\ldots, \ol {q_n}]^t,$$
where $t$ stands for transpose.

Thus $(\ol {q_1},\ldots,\ol {q_n})$ is a basis of the free module
$P/JP$. Let $w_J$ be given by the set of generators $\ol
{\gl(q_1)},\ldots,\ol {\gl(q_n)}$ of $J/J^2$, i.e $w_J:(A/J)^n \surj
J/J^2$ given by $w_J(e_i)=\ol {\gl(q_i)}$ for $i=1,\ldots,n$.

We define $e(P,(p_1,\ldots,p_{n+r}),\sigma)=(J,w_J) \in E^n(A)$. We
need to show that $e(P,(p_1,\ldots,p_{n+r}),\sigma)$ is independent of
the choice of the elementary completion of $\ol \sigma$.

\begin{lemma}
Suppose $\Gamma\in E_{n+r}(A/J)$ is chosen so that its first $r$ rows
are $\ol \sigma$. Let $\Gamma [\ol {p_1},\ldots,\ol
{p_{n+r}}]^t=[0,\ldots,0,\ol {q_1'},\ldots,\ol {q_n'}]^t$. Then there
exists $\Psi \in E_n(A/J)$ such that $\Psi [\ol {q_1},\ldots, \ol
{q_n}]^t=[\ol {q_1'},\ldots,\ol {q_n'}]$.
\end{lemma}

\begin{proof}
The matrix $\Gamma \Delta^{-1} \in E_{n+r}(A/J)$ is such that its
first $r$ rows are $\left [
	\begin{array}{cccccc}
	 1& \ldots & 0 &0 &\ldots &0\\ \vdots &&\vdots &\vdots &
	   & \vdots \\ 0 & \ldots &1 &0 &\ldots & 0
	\end{array} \right].$ 
Therefore, there exists $\Psi\in \SL_n(A/J) \cap E_{n+r}(A/J)$ such
that $\Psi [\ol {q_1},\ldots,\ol {q_n}]^t = [\ol {q_1'},\ldots, \ol
{q_n'}]^t$.
Since $n > \dim A/J +1$, by (\cite{vas}, Theorem 3.2), $\Psi \in E_n(A/J)$.
$\hfill \gj$
\end{proof}

The remaining arguments needed to show that
$e(P,(p_1,\ldots,p_{n+r}),\sigma)$ is a well defined element of
$E^n(A)$ is same as in (\cite{Bh-Raja-5}, p. 152-153), hence we omit
it. Therefore we have a well defined map $\Um_{r,n+r}(A) \by e E^n(A)$.

The following result can be proved by following the proof of
(\cite{Bh-Raja-5}, Theorem 5.4). Hence we omit the proof.

\begin{theorem}\label{stably}
Let $A$ be a regular ring of dimension $d$ containing an infinite
field $k$ and let $n$ be an integer such that $2n \geq d+3$. Let $P$
be a stably free $A$-module of rank $n$ defined by $\sigma \in
\Um_{r,n+r}(A)$. Then $P$ has a unimodular element if and only if
$e(P)=e(\sigma)=0$ in $E^n(A)$.
\end{theorem}


\subsection{Whitney class homomorphism}

Let $A$ be a regular domain of dimension $d\geq 2$ containing an
infinite field $k$ and let $Q$ be a stably free $A$-module of rank $n$
with $2n\geq d+3$. In (\ref{stably}), we proved that $e(Q)=0$ in
$E^n(A)$ if and only if $Q$ has a unimodular element.  Using this
result we will establish a whitney class homomorphism of stably free
modules. When $n+s=d$, then (\ref{W}) is proved in
(\cite{SY}, Theorem 3.1) for any projective $A$-module $Q$. Our proof
is a simple adaptation of their proof.

\begin{theorem}\label{W}
Let $A$ be a regular domain of dimension $d\geq 2$ containing an
infinite field $k$. Suppose $Q$ is a stably free $A$-module of rank
$n$ defined by $\sigma\in \Um_{r,n+r}(A)$. Then there exists a homomorphism
$w(Q):E^s(A)\ra E^{n+s}(A)$ for every integer $s\geq 1$
with $2n+s\geq d+3$.
\end{theorem}

\begin{proof}
Write $F=A^n$ and $F'=A^s$. Let $I$ be an ideal of height $s$ and
$w:F'/IF' \surj I/I^2$ be an equivalence class of surjective
homomorphisms, where the equivalence is defined by $E_s(A/I)=E(F'/IF')$
maps. To each such pair $(I,w)$, we will associate an element
$w(Q)\cap (I,w)\in E^{n+s}(A)$.

First we can find an ideal $\wt I\subset A$ of height $\geq n+s$ and a
surjective homomorphism $\psi:Q/IQ \surj \wt I/I$ (this is just the
existence of a generic surjection of $Q/IQ$). Let $\psi \ot A/\wt
I=\wt \psi$. Then $\wt \psi :Q/\wt IQ \surj \wt I/(I+\wt I^2)$ is a
surjection.

Since $\dim A/\wt I\leq d-(n+s)\leq n-3$, $Q/\wt IQ$ is a free $A/\wt
I$-module, by Bass result (\cite{B}). Let ``bar'' denotes reduction
modulo $\wt I$, then $\ol \sigma \in \Um_{r,n+r}(\ol A)$ can be
completed to an elementary matrix $\Theta \in E_{n+r}(\ol A)$. This
gives a well defined basis $[\ol q_1,\ldots,\ol q_n]$ for $\ol Q$
which does not depends on the elementary completions of $\ol \sigma$
(in the sense that any two basis of $\ol Q$ obtained this way will be
connected by an element of $E_n(\ol A)$).

Let $\gamma:F/\wt IF \iso Q/\wt I Q$ be the isomorphism given by
$\gamma(\ol {e_i})=\ol q_i$ for $i=1,\ldots,n$, where $e_1,\ldots,e_n$ is
the standard basis of the free module $F$. Let $\gb=\wt \psi
\gamma :F/\wt IF \surj \wt I/(I+\wt I^2)$ be a surjection and let
$\gb' :F/\wt IF \ra \wt I/\wt I^2$ be a lift of $\gb$.

Further, $w:F'/IF' \surj I/I^2$ induces a surjection $\wt w:F'/\wt IF'
\surj (I+\wt I^2)/\wt I^2$. Composing $\wt w$ with the natural
inclusion $ (I+\wt I^2)/\wt I^2 \subset \wt I/\wt I^2$, we get
a map $w': F'/\wt IF' \ra \wt I/\wt I^2$.

Combining $w'$ and $\gb'$, it is easy to see that we get a surjective
homomorphism $$\gD=\gb'\op w' : F/\wt IF\op F'/\wt IF' = (F\op F')/\wt I
(F\op F') \surj \wt I/\wt I^2$$ (surjectivity follows by considering
the exact sequence $0\ra (I+\wt {I^2})/\wt {I^2} \inj \wt I/\wt {I^2}
\ra \wt I/(I+\wt {I^2})\ra 0$). We have $(\wt I,\gD)$ a local
orientation of $\wt I$. We will show that the image of $(\wt I,\gD)$
in $E^{n+s}(A)$ is independent of choices of $\psi$, the lift $\gb'$ and
the representative of $w$ in the equivalence class.\\

{\bf Step 1.} First we show that for a fixed $\psi$, $(\wt I,\gD)$ in
$E^{n+s}$ is independent of the lift $\gb'$ and the representative of
$w$.

$(a)$ Suppose $w,w_1 :F'/IF' \surj I/I^2$ are two equivalent local
orientations of $I$. Then $w_1=w\epsilon$ for some $\epsilon \in
E(F'/IF')$. Using the canonical homomorphisms $E(F'/IF')\surj E(F'/\wt
IF') \ra E((F\op F')/\wt I(F\op F'))$, we get that $w_1'=w'\epsilon_1$
for some $\epsilon_1\in E((F\op F')/\wt I(F\op F'))$. 

Let $\gD_1$ be the local orientation of $\wt I$ obtained by using
$\gb'$ and $w_1$. Then $\gD_1=\gD \epsilon_1$. Hence $(\wt
I,\gD)=(\wt I,\gD_1)$ in $E^{n+s}(A)$.

$(b)$ Let $\gb'' :F/\wt IF \ra \wt I/ \wt I^2$ be another lift of
$\gb$. Then $\phi=\gb'-\gb'' :F/\wt IF \ra (I+\wt I^2)/\wt I^2$. Since
$\wt w_1:F'/\wt IF' \surj (I+\wt I^2)/\wt I^2$ is a surjection, there
exists $g:F/\wt IF \ra F'/\wt IF'$ such that $\wt w_1 g=\phi$.

Let $\epsilon_2=\left(
	\begin{smallmatrix}
	1 & 0 \\
	g & 1
	\end{smallmatrix}
\right) \in E((F\op F')/\wt I(F\op F'))$. Then $(\gb'' \op
w_1')\epsilon_2 = (\gb'\op w_1')$. Therefore, if $\gD_2=\gb''\op
w_1'$, then $\gD_2 \epsilon_2 =\gD_1=\gD\epsilon_1$.

This  completes the proof of the claim in step 1.\\

{\bf Step 2.} Now we will show that $(\wt I,\gD)\in E^{n+s}(A)$ is
independent of $\psi$ also (i.e. it depends only on $(I,w)$).

Recall that $w:F'/IF' \surj I/I^2$ is a surjection. It is easy to see
that we can lift $w$ to a surjection $\Omega:F' \surj I\cap K$, where
$K+I=A$ and $K$ is an ideal of height $s$ (or $K=A$).

We can find an ideal $\wt K\subset A$ of height $\geq n+s$ and a
surjective homomorphism $\psi':Q/KQ \surj \wt K/K$. Let $\psi'\ot
A/\wt K =\wt {\psi'}$. Then $\wt {\psi'} :Q/\wt KQ \surj \wt K/(K+\wt
K^2)$ is a surjection.

Again, since $\dim A/\wt K \leq n-3$, $Q/\wt K Q$ is a free $A/\wt
K$-module. If ``bar'' denotes reduction modulo $\wt K$, then $\ol
\sigma \in \Um_{r,n+r}(A/\wt K)$ can be completed to an elementary
matrix which gives a basis $\ol p_1,\ldots,\ol p_n$ for $Q/\wt
KQ$. Let $\gamma':F/\wt KF \iso Q/\wt KQ$ be the isomorphism given by
$\gamma'(\ol {e_i})=\ol p_i$. Let $\eta=\wt {\psi'}\gamma' :F/\wt KF \surj
\wt K/(I+\wt K^2)$ be a surjection and let $\eta':F/\wt KF \ra \wt
K/\wt K^2$ be a lift of $\eta$.

The map $\Omega : F'\surj I\cap K$ induces a surjection $\Omega\ot
A/K=\Omega':F'/KF' \surj K/K^2$ which in turn induces a surjection
$\Omega'\ot A/\wt K={w''} : F'/\wt KF' \surj (K+\wt K^2)/\wt
K^2$. Since $(K+\wt K^2)\subset \wt K$, we get a map $w'' :F'/\wt KF'
\ra \wt K/\wt K^2$.

Combining $w''$ and $\eta'$, we get a surjection $\Delta'=\eta'\op w''
:(F\op F')/\wt K(F\op F') \surj \wt K/\wt K^2$.\\

{\bf Claim.} $(\wt I,\gD)+(\wt K,\gD')=0$ in $E^{n+s}(A)$. \\

Since $I+K=A$, we get $\wt I+\wt K=A$. Further, we get a surjection
$$\Psi=\psi\op \psi' :Q/(I\cap K)Q \simeq Q/IQ\op Q/KQ \surj \wt I/I
\op \wt K/K \simeq (\wt I\cap \wt K)/(I\cap K).$$ Let $\wt \Psi:Q\ra
\wt I\cap \wt K$ be a lift of $\Psi$ such that the following holds:\\

$(i)$ $\wt \Psi \ot A/\wt I= \wt \psi$, where $\wt \psi:Q/\wt IQ \surj
\wt I/(I+\wt I^2)$ is a surjection and

$(ii)$ $\wt \Psi \ot A/\wt K=\wt {\psi'}$, where $\wt {\psi'} : Q/\wt
KQ \surj \wt K/(K+\wt K^2)$ is a surjection.\\

Let $\wt \Psi_1 : Q/\wt I Q \ra \wt I/\wt I^2$ be a lift of $\wt \Psi
\ot A/\wt I$ and let $\wt \Psi_2 :Q/\wt KQ \ra \wt K/\wt K^2$ be a
lift of $\wt \Psi \ot A/\wt K$. Then $\wt \Psi_1$ and $\wt \Psi_2$
induces a map $\wt \Psi_3 : Q/(\wt I\cap \wt K)Q \ra (\wt I\cap \wt
K)/(\wt I\cap \wt K)^2$.

Since $\gb=\wt \psi \gamma = (\wt \Psi \ot A/\wt I)\gamma$ and
$\gb':F/\wt IF \ra \wt I/\wt I^2$ is a lift of $\gb$, we get that
$\ga_1=\gb' \gamma^{-1} - \wt \Psi_1$ is a map from $Q/\wt IQ$ to
$(I+\wt I^2)/\wt I^2 \subset \wt I/\wt I^2$. Similarly, $\ga_2 = \eta'
(\gamma')^{-1} - \wt \Psi_2$ is a map from $Q/\wt KQ$ to $(K+\wt
K^2)/\wt K^2 \subset \wt K/\wt K^2$.

Since $\wt w:F'/\wt IF' \surj (I+\wt I^2)/\wt I^2$ is a surjection, we
can find $g_1 : Q/\wt IQ \ra F'/\wt IF'$ such that $\wt w
g_1=\ga_1$. Similarly, we can find $g_2:Q/\wt KQ \ra F'/\wt KF'$ such
that $w'' g_2=\ga_2$ (here $w''=\Omega'\ot A/\wt K$).

Let $g$ be given by $g_1,g_2$ and $\wt \gamma$ be given by
$\gamma,\gamma'$. Then \\

$(a)$ $\left(
	\begin{smallmatrix}
	\wt \gamma & 0 \\
	0 & 1
	\end{smallmatrix}
\right)$ is an isomorphism from $({F\op F'})/{(\wt I\cap \wt
K)(F\op F')}$ to $({Q\op F'})/{(\wt I\cap \wt K)(Q\op F')}$ and

$(b)$ $\left(
	\begin{smallmatrix}
	1 & 0 \\
	g & 1
	\end{smallmatrix}
\right)$ is an automorphism of $({Q\op F'})/{(\wt I\cap \wt
K)(Q\op F')}.$\\

Write $\Gamma=\left(
	\begin{smallmatrix}
	1 & 0 \\
	g & 1
	\end{smallmatrix}
\right) \left(
	\begin{smallmatrix}
	\wt \gamma & 0 \\
	0 & 1
	\end{smallmatrix}
\right)$.
Since $\wt \Psi$ is a lift of $\Psi$, $\Psi$ is a surjection from
$Q/(I\cap K)Q$ to $(\wt I\cap \wt K)/(I\cap K)$ and $\Omega : F' \surj
I\cap K$ is a surjection, we get that $\wt \Psi \op \Omega : Q\op F'
\surj \wt I\cap \wt K$ is a surjection. 

Write $\Theta =(\wt \Psi \op \Omega) \ot A/(\wt I\cap \wt K)$. Then
$\Theta :({Q\op F'})/{(\wt I\cap \wt K)(Q\op F')} \surj ({\wt I\cap
\wt K})/{(\wt I\cap \wt K)^2}$. Let $(\Delta,\Delta') : ({F\op
F'})/{(\wt I\cap \wt K)(F\op F')} \surj ({\wt I\cap \wt K})/{(\wt
I\cap \wt K)^2}$ be the surjection induced from $\Delta,\Delta'$.  We
claim that $ (\Delta,\Delta')=\Theta \Gamma$. (This follows by
checking on $V(\wt I)$ and $V(\wt K)$ separately, but we give a direct
proof below.)

Let $\ga_3 :Q/(\wt I\cap \wt K)Q \ra (\wt I \cap \wt K)/(\wt I\cap \wt
K)^2$ be the map induced from $\ga_1,\ga_2$ and let $\tau:F/(\wt I\cap
\wt K) \ra (\wt I \cap \wt K)/(\wt I\cap \wt K)^2$ be the map induced
from $\gb',\eta'$. Then we have $\ga_3=\tau \wt\gamma^{-1} -
\wt\Psi_3$. Let $\ol \Omega : F'/(\wt I\cap \wt K)F' \ra (\wt I \cap
\wt K)/(\wt I\cap \wt K)^2$ be the map induced from $\wt w,w''$. Then
we have $\ol \Omega g=\ga_3$.

Now $\Theta
\Gamma (0,y)=\Theta(0,y)=\ol \Omega(y)=(\Delta,\Delta')(0,y)$ and
$\Theta\Gamma (x,0)=\Theta (\wt \gamma(x),g\wt \gamma(x)) =\wt
\Psi_3 \wt \gamma (x) + \ol \Omega g \wt \gamma (x) =\wt \Psi_3\wt
\gamma (x) + \tau \wt \gamma^{-1}\wt\gamma (x)-\wt\Psi_3 \wt\gamma (x)
= \tau(x)=(\gD,\gD')(x,0).  $

This proves that $(\Delta,\Delta')=\Theta \Gamma$.
By (\cite{Bh-Raja-5}, Theorem 4.2), we get that $(\wt I,\gD)+(\wt
K,\gD')=0$ in $E^{n+s}(A)$. Since $(\wt K,\gD')$ depends only on
$(I,w)$, it follows that $(\wt I,\gD)$ is independent of the choice of
$\psi$. This establishes the claim in step 2.\\

If $(I,w)$ is a global orientation, then we can take $K=A$ in the
above proof and it will follow that $(\wt I,\gD)$ is also a global
orientation.

Thus the association $(I,w)\mapsto (\wt I,\gD)\in E^{n+s}(A)$ defines
a homomorphism $\phi(Q) : G^s(A)\ra E^{n+s}(A)$, where $(I,w)$
are the free generators of $G^s(A)$. Further $\phi(Q)$ factors through a
homomorphism $w(Q) : E^s(A)\ra E^{n+s}(A)$ sending
$(I,w)\in E^s(A)$ to $(\wt I,\gD)\in E^{n+s}(A)$.
This completes the proof of the theorem.  $\hfill \gj$
\end{proof}

\begin{corollary}\label{W1}
Let $A$ be a regular domain of dimension $d\geq 2$ containing an
infinite field. Suppose $Q$ is a stably free $A$-module of rank
$n$. Then there exists a homomorphism $w_0(Q):E_0^s(A) \ra
E_0^{n+s}(A)$ for every integer $s\geq 1$ with $2n+s\geq d+3$.
\end{corollary}

\begin{proof}
The proof is similar to that of (\ref{W}) and we give an
outline. Write $F=A^n$ and $F'=A^s$.

Suppose $(I)$ is a generator of $G_0^s(A)$. Here $I$ is an ideal of
height $s$, $\Spec (A/I)$ is connected and there is a surjection from
$F'/IF'$ to $I/I^2$. There is a surjection $\psi : Q/IQ \surj \wt
I/I$, where $\wt I$ is an ideal of height $\geq n+s$. For such a
generator $(I)$, we associate $(\wt I)\in E_0^{n+s}(A)$.

For well-definedness, fix a local orientation $w:F'/IF' \surj I/I^2$
and a surjective lift $\Omega:F' \surj I\cap K$ of $w$, where $K$ is
an ideal of height $\geq s$ and $K+I=A$. Let $\psi':Q/KQ \surj \wt
K/K$ be a surjection, where $\wt K$ is an ideal of height $\geq
n+s$. As in (\ref{W}), there exists a surjection from $F\op F' \surj
\wt I\cap \wt K$. This shows that $(\wt I)+(\wt K)=0$ in
$E_0^{n+s}(A)$ and so $(\wt I)\in E_0^{n+s}(A)$ is independent of the
choice of $\psi$.

The association $(I)\mapsto (\wt I)\in E_0^{n+s}(A)$ extends to a
homomorphism $\phi_0 :G_0^s(A) \ra E_0^{n+s}(A)$.

If $(I)$ is global (i.e. $I$ is a surjective homomorphism of $F'$),
then taking $K=A$ in the above argument, we can prove that $(\wt I)$
is also global. So $\phi_0$ factors through a homomorphism
$w_0(Q):E_0^s(A) \ra E_0^{n+s}(A)$.
$\hfill \gj$
\end{proof}

\begin{definition}
The homomorphism $w(Q)$ in theorem \ref{W} will be called the {\it
Whitney class homomorphism}. The image of $(I,w)\in E^s(A)$ under
$w(Q)$ will be denoted by $w(Q)\cap (I,w)$.

Similarly, the homomorphism $w_0(Q)$ in (\ref{W1}) will be called the
{\it weak Whitney class homomorphism}. The image of $(I)\in E_0^s(A)$
under $w_0(Q)$ will be denoted by $w_0(Q) \cap (I)$.
\end{definition}

The proof of the following result is same as (\cite{SY}, Corollary
3.4), hence we omit it.

\begin{corollary}
Let $A$ be a regular domain of dimension $d\geq 2$ containing an
infinite field. Suppose $Q$ is a stably free $A$-module of rank
$n$. For every integer $s\geq 1$ with $2n+s\geq d+3$, we have
$$w_0(Q)\zeta^s=\zeta^{n+s}w(Q)\;\; {\rm and}\;\;
C^n(Q^*)\eta^s=\eta^{n+s}w_0(Q),$$ where $(i)$ $\zeta^r :E^r(A)\surj
E_0^r(A)$ is a natural surjection obtained by forgetting the
orientation,

$(ii)$ $\eta^r:E_0^r(A) \ra CH^r(A)$ is a natural homomorphism,
sending $(I)$ to $[A/I]$. Here $CH^r(A)$ denotes the Chow group of
cycles of codimension $r$ in $\Spec (A)$ and

$(iii)$ $C^n(Q^*)$ denote the top Chern class homomorphism \cite{F}.
\end{corollary}

The following result is about vanishing of Whitney class
homomorphism. When $n+s=d$, it is proved in (\cite{SY}, Theorem 3.5)
for arbitrary projective module $Q$ and our proof is an adaptation
of \cite{SY}. We will follow the proof of (\ref{W}) with necessary
modifications.

\begin{theorem}\label{W2}
Let $A$ be a regular domain of dimension $d\geq 2$ containing an
infinite field. Suppose $Q$ is a stably free $A$-module of rank $n$
defined by $\sigma \in \Um_{r,n+r}(A)$. Let $s\geq 1$ be an integer
with $2n+s\geq d+3$. Write $F=A^n$ and $F'=A^s$. Let $I$ be an ideal
of height $s$ and let $w:F'/IF' \surj I/I^2$ be a surjection. If
$Q/IQ=P_0 \op A/I$, then $w(Q)\cap (I,w)=0$ in $E^{n+s}(A)$.

In particular, if $Q=P\op A$, then the homomorphism $w(Q):E^s(A) \ra
E^{n+s}(A)$ is identically zero. Similar statements hold for $w_0(Q)$.
\end{theorem}

\begin{proof}
{\bf Step 1.}
We can find an ideal $\wt I\subset A$ of height $n+s$ and a surjective
homomorphism $\psi:Q/IQ \surj \wt I/I$. Let $\wt \psi =\psi \ot A/\wt
I:Q/\wt I \surj \wt I/(I+\wt I^2)$.

Let $\Omega:F'\ra I$ be a lift of $w$ and let $\ol w=w\ot A/\wt I :
F'/\wt IF' \surj I/I\wt I$. Composing $\ol w$ with the natural map
$I/I\wt I \inj \wt I/I\wt I \surj \wt I/\wt I^2$, 
we get a map $w':F'/\wt IF' \ra \wt I/\wt I^2$.

Since $Q/IQ=P_0\op A/I$, we can write $\psi =(\theta,\ol a)$ for some
$a\in \wt I$ and $\theta \in P_0^*$. We may assume that $\psi(P_0)=\wt
J/I$, for some ideal $\wt J\subset A$ of height $n+s-1$. Note that
$\wt I=(\wt J,a)$.

Since $\dim A/\wt J=d-(n+s-1)\leq n-2$ and $P_0/IP_0$ is stably free
$A/I$-module of rank $n-1$, $P_0/\wt JP_0$ is free.
If ``prime'' denotes reduction modulo $\wt J$, then $\sigma'$ can be
completed to an elementary matrix in $E_{n+r}(A/\wt J)$. This
gives a canonical basis of $P_0/\wt JP_0$, say $q_1',\ldots,q_{n-1}'$. 
Let $\gamma' :(A/\wt J)^{n-1} \iso P_0/\wt JP_0$ be the isomorphism
given by $[q_1',\ldots,q_{n-1}']$.

Let $\gamma :F/\wt IF=(A/\wt I)^n \iso Q/\wt IQ =P_0/\wt IP_0 \op A/\wt
I$ be the isomorphism given by $(\gamma',1)$, i.e. $\gamma=[\ol
{q_1},\ldots,\ol {q_{n-1}},1]$. Let $\gb=\wt \psi \gamma : F/\wt IF
\surj \wt I/(I+\wt I^2)$ and let $\gb' : F/\wt IF \ra \wt I/\wt I^2$
be a lift of $\gb$.

As in the proof of (\ref{W}), combining $w'$ and $\gb'$, we get a
surjection $\gD=\gb'\op w' : (F\op F')/\wt I(F\op F') \surj \wt I/\wt
I^2$ and $(\wt I,\gD)=w(Q)\cap (I,w)$. We claim that $(\wt I,\gD)=0$ in
$E^{n+s}(A)$.\\

{\bf Step 2.} In this step, we will prove the claim.
The surjection $\gt:P_0 \surj \wt J/I$ induces a surjection $\ol \gt
=\gt \ot A/\wt J : P_0/\wt JP_0 \surj \wt J/(I+\wt J^2)$. Let
$\zeta=\ol \gt \gamma' : (A/\wt J)^{n-1} \surj \wt J/(I+\wt J^2)$ and
let $\zeta' : (A/\wt J)^{n-1} \ra \wt J/\wt J^2$ be a lift of $\zeta$.

If $\ol {\zeta'}$ denotes the composition of $\zeta'\ot A/\wt I :
(A/\wt I)^{n-1} \ra \wt J/\wt J\wt I$ with natural maps $\wt J/\wt
J\wt I\inj \wt I/\wt J\wt I \surj \wt I/\wt I^2$, we get that $(\ol
{\zeta'},\ol a)$ is a lift of $\gb:F/\wt IF \surj \wt I/(I+\wt
I^2)$. Since $w(Q)\cap (I,w)$ is independent of the lift $\gb'$ of
$\gb$, we may assume that $\gb'=(\ol {\zeta'},\ol a)$.

If $\gd: A^{n-1} \ra \wt J$ is a lift of $\zeta'$, then
$(\gd,a,\Omega):F\op F' \ra \wt I$ is a lift of $(\gb',w')$.  If $\wt
{J'}$ is the image of $(\gd,\Omega)$, then $\wt J=\wt {J'}+\wt
J^2$. (To see this, let $y\in \wt J$, then there exists $x\in A^{n-1}$
such that $\gd(x)-y=y_1+z$ for some $y_1\in I$ and $z\in \wt
J^2$. Choose $x_1\in F'$ such that $y_1-\Omega(x_1)=z_1 \in I^2
\subset \wt J^2.$ Therefore $\gd(x)-\Omega(x_1) =y$ modulo $\wt J^2$.)

Since $\wt J=\wt {J'}+\wt J^2$, we can find $e\in \wt J^2$ such that
$(1-e)\wt J \subset \wt {J'}$ and $\wt J=(\wt {J'},e)$. Therefore by
(\cite{MK}, Lemma 1), $\wt I=(\wt J,a)=(\wt {J'},b)$, where
$b=e+(1-e)a$. Thus $(\gd,b,\Omega) : F\op F' \surj \wt I$ is a
surjection which is a lift of $\gb'\op w'$. This proves that $(\wt
I,\gD)=0$ in $E^{n+s}(A)$. This completes the proof.
$\hfill \gj$
\end{proof}

 
\subsection{Remark on some results of Yang}

We start this section by describing some results of Yang \cite{Y}.

$(1)$ Let $R$ be a Noetherian commutative ring of dimension $d$ and let $n$
be an integer with $2n\geq d+3$.  Let $l$ be an ideal of $R$ and let
$\rho:R\surj \ol R=R/l$ be the natural surjection. Yang \cite{Y}
defines a group homomorphism $E(\rho):E^n(R;R) \ra E^n(\ol R;\ol R)$, called
the restriction map of Euler class group, as $E(\rho)(I,w_I)=(\ol
{I'+l},w_{\ol{I'+l}})$, where $(I,w_I)=(I',w_{I'})$ in $E^n(R;R)$ with
height of $\ol {I'+l} \geq n$ in $\ol R$.

$(2)$ Further, let $A$ be a Noetherian commutative ring of dimension $s$
with $2n\geq s+3$. Assume there exists a ring homomorphism $\phi:R\ra
A$ such that for any local orientation $(I,w_I)\in E^n(R;R)$, height
of $\phi(I)$ is $\geq n$. Then Yang defines a group homomorphism
$E(\phi):E^n(R;R) \ra E^n(A;A)$, called the extension map of Euler class
group, as $E(\phi)(I,w_I)=(\phi(I),w_{\phi(I)})$.

$(3)$ Let $D(R,l)$ denotes the double of $R$ along $l$, then $$0\ra l\ra
D(R,l)\by {p_1} R\ra 0$$ is a split exact sequence.  The relative Euler
class group of $R$ and $l$ is defined as $$E^n(R,l;R)=
ker(E(p_1):E^n(D(R,l);D(R,l)) \ra E^n(R;R)),$$ where $E(p_1)$ is the
restriction map.

$(4)$ (Homology sequence) Let $p_2$ denote the second projection from
$D(R,l)\ra R$. Then the following Homology sequence of Euler class
group is exact.
$$E^n(R,l;R)\lby {E(p_2)} E^n(R;R) \lby {E(\rho)} E^n(R/l;R/l).$$ 

$(5)$ (Excision theorem) Further assume that there exists a splitting
of $\rho : R\ra R/l$ (i.e. a ring homomorphism $\gb:R/l\ra R$ such
that $\rho \gb=id$) satisfying the condition that for any local
orientation $(J,w_J)\in E^n(R/l;R/l)$, height of $\gb(J)$ is $\geq
n$. Then we have the following exact sequence, called the Excision
sequence of Euler class group.
$$0\ra E^n(R,l;R)\lby {E(p_2)} E^n(R;R) \lby {E(\rho)} E^n(R/l;R/l)\ra 0.$$

Note that the existence of a splitting of $\rho$ is sufficient for the
injectivity of Homology sequence.  

$(6)$ As a consequence of above
results, if $2n\geq d+4$, then we have the following split short exact
sequence:
$$0\ra E^n(R[T],(T);R[T])\lby {E(p_2)} E^n(R[T];R[T]) \lby {E(\rho)}
E^n(R;R)\ra 0.$$ Further if $R$ is a regular affine domain essentially
of finite type over an infinite perfect field, then it is proved that
$E(\rho):E^n(R[T];R[T])\ra E^n(R;R)$ is an isomorphism. Note that Das
and Raja (\cite{DR}, Theorem 3.8) also proved this isomorphism for
$2n\geq d+3$.

Using (\cite{Bh-Manoj}, 4.11, 5.7) and following the proof in \cite{Y}, we
get the following stronger results in case of polynomial ring over
$R$.

\begin{theorem}
Assume that $R$ is a regular domain of dimension $d$ containing an
infinite field and let $n$ be an integer with $2n\geq d+3$. Then we
have the following results:

$(i)$ (Homology sequence) Let $p_2$ denote the second projection from
$D(R[T],l)\ra R[T]$, where $l$ is an ideal of $R[T]$. Then we have the
following Homology exact sequence of Euler class group:
$$E^n(R[T],l;R[T])\lby {E(p_2)} E^n(R[T];R[T]) \lby {E(\rho)}
E^n(R[T]/l;R[T]/l).$$

$(ii)$ (Excision theorem) Further assume that there exists a splitting $\beta$
of $\rho : R[T]\ra R[T]/l$ 
satisfying the condition that for any local
orientation $(J,w_J)\in E^n(R[T]/l;R[T]/l)$, height of $\gb(J)$ is $\geq
n$. Then we have the following Excision exact
sequence of Euler class group:
$$0\ra E^n(R[T],l;R[T])\lby {E(p_2)} E^n(R[T];R[T]) \lby {E(\rho)}
E^n(R[T]/l;R[T]/l)\ra 0.$$

In particular, when $l=(T)$, then we have the following split short
exact sequence:
$$0\ra E^n(R[T],(T);R[T])\lby {E(p_2)} E^n(R[T];R[T]) \lby {E(\rho)}
E^n(R;R)\ra 0.$$

\end{theorem}


{}

\end{document}